\documentclass[titlepage,12pt]{article}
\usepackage{amsmath,amsthm,amssymb}
\usepackage[mathscr]{eucal}
\usepackage[utf8]{inputenc}
\usepackage{graphicx}
\usepackage{bm}
\usepackage[comma,compress]{natbib}

\usepackage{color}

\usepackage[english]{babel}
\usepackage{amsmath,amstext,amssymb}
\usepackage{array}

\title{Comparisons of coherent systems under the Time 
	Transformed Exponential model}
\author{Jorge Navarro\\Dpto. Estad\'istica e Investigaci\'on Operativa, Universidad de Murcia
\\ Facultad de Matem\'aticas, Campus de Espinardo \\
30100 Murcia, SPAIN\\
ORCID ID: 0000-0003-2822-915X\\ \texttt{jorgenav@um.es}
\and Julio Mulero \\Dpto. Matemáticas, Universidad de
Alicante \\ Facultad de Ciencias, Apartado de correos 99 \\ 03080
Alicante, SPAIN \\
ORCID ID:  0000-0001-5949-7611 
\\
\texttt{julio.mulero@ua.es}
}
\date{}

\newtheorem{theorem}{Theorem}[section]
\newtheorem{proposition}[theorem]{Proposition}
\newtheorem{definition}[theorem]{Definition}

\newtheorem{corollary}[theorem]{Corollary}
\newtheorem{example}[theorem]{Example}

\newtheorem{remark}[theorem]{Remark}

\numberwithin{equation}{section} \numberwithin{theorem}{section}

\def\ov{\overline}

\begin{document}

\maketitle
\begin{abstract} 
The coherent systems are basic concepts in reliability theory and survival analysis. They contain as particular cases the popular series, parallel and $k$-ou-of-$n$ systems (order statistics). Many results have been obtained for them by assuming that the component lifetimes are independent. In many practical cases, this assumption is unrealistic. In this paper we study them by assuming a Time Transformed Exponential (TTE) model for the joint distribution of the  component lifetimes. This model is equivalent to the frailty model which assumes that they are conditionally independent given a common risk parameter (which represents the common environment risk). Under this model, we obtain explicit expressions for the system reliability functions and comparison results for the main stochastic orders. The system residual lifetime (under different assumptions) is studied as well. 
\\
%
%\vspace*{1cm}
\end{abstract}

\newpage
\section{Introduction}

The real engineering or biological  systems can be represented theoretically by using coherent system mathematical structures. They can also be used to represent networks. Thus, a {\it system} is just a Boolean function $\psi:\{0,1\}^n\to\{0,1\}$ which indicates if the system works $\psi(x_1,\dots,x_n)=1$ (or fails $\psi(x_1,\dots,x_n)=0$) when we know the states of the components represented by $x_1,\dots,x_n$. A system $\psi$ is {\it coherent} if $\psi$ is increasing and all the components are relevant for the system.  Its formal definition and basic properties can be found in the classic book \cite{BP75}. 

There exists a wide literature on coherent systems. Many results are obtained by assuming  independent components. However, in practice, sometimes this assumption is unrealistic. Recently, some results have been extended to the case of exchangeable components (by using the concept of signature) and to the general case of arbitrary dependent component lifetimes (by using a copula approach). A recent review of these results can be seen in \cite{N18a}. Additional recent results can be seen in \cite{KA16,PAD17,KN18,N18b} and the references therein.  

The Time Transformed Exponential (TTE) model for a joint distribution 
is an interesting dependence model for the components lifetimes of a system. It is equivalent to assume a frailty model which assumes that they are conditionally independent given a common parameter $\theta$ which represents the risk due to the common environment. The very well known proportional hazard rate model is assumed to measure the influence of parameter $\theta$ in the component reliability functions.
The different dependence models are obtained just by changing the a priori distribution of parameter $\theta$. This model is also equivalent to assume strict Archimedean copulas for the joint distribution. The formal definition of this model is given  in Section 2. For the basic properties of this model we refer the reader to \cite{BS05,N06,MP10,MPR10a,MPR10b} and the references therein.

In this paper we study coherent systems by assuming that their component lifetimes are dependent and satisfy a TTE model. Under this assumption, we provide explicit expressions for the system reliability functions. These expressions are used to compare the different systems or the same system under different dependence models by using the main stochastic orders. Partidular results are obtained in the case of identically distributed (ID) component lifetimes. 
The same is done for the system residual lifetime under different assumptions. Specifically we consider two cases, when at time $t>0$ we just know that the system is working, or when we know that all the components are working.

The rest of the paper is orginized as follows. In Section 2 we introduce the notation and give the basic results needed in the paper (including the formal definition of the TTE model). The main representations for coherent system reliabilities under TTE models are given in Section 3. The stochastic comparisons  are placed in Section 4 and the results for residual lifetimes in Section 5. Illustrative examples are included in Sections 3--5. Finally, some conclusions are given in Section 6.

%To be written at the end.\\

\section{Preliminary results}

Let $\mathbf{X}=(X_1,\dots,X_n)$ be a nonnegative random
vector (representing the component lifetimes) with joint reliability (or survival) function 
\begin{equation*}
 \mathbf{\ov F}(x_1,\dots,x_n)= \Pr(X_1>x_1,\dots,X_n>x_n),
\end{equation*}
and marginal univariate reliability (survival) functions given by
\[
\ov F_i(x_i)=\mathbf{\ov F} (0,\dots,0,x_i,0,\dots,0)=\Pr(X_i>x_i)
\]
where we assume that $\ov F_i(0)=1$ for $i=1,\dots,n$.

As pointed out in the literature (see, e.g., Nelsen, 2006), the
dependence structure of the random vector $\mathbf{X}$ can be
described by its \textit{survival copula} $K$ which allows us to represent $\mathbf{\ov F}$ as
\begin{equation*}
\mathbf{\ov F}(x_1,\dots,x_n)= K(\ov F_1(x_1),\dots,\ov F_n(x_n))
\end{equation*}
for all $x_1,\dots,x_n$. The function $K$ is a copula, that is,  a continuous multivariate distribution function with uniform marginals over the interval $(0,1)$.
From Sklar's theorem (see Nelsen 2006, p. 47), the survival copula is
unique under the assumption that 
the marginal reliability  functions $\ov F_1,\dots,\ov F_n$ are continuous. 

The representation based on survival copulas is used in reliability and actuarial
sciences instead of that of  the ordinary (distributional) copulas,  since in that areas, reliability or survival functions are commonly studied instead of cumulative
distribution functions. Any copula function can be used both as a survival copula or a distributional copula. However, the probability models obtained can be different.

Among survival copulas, particularly interesting is the class of strict
Archimedean survival copulas. Recall that a copula $K$ is said to be
\textit{strict Archimedean} (see, e.g.,  Nelsen 2006, p. 112), if it can be written as
\begin{equation}\label{2-1}
K(u_1,\dots,u_n)=W(W^{-1}(u_1)+\dots+W^{-1}(u_n)) \ 
\end{equation}
for all  $u_1,\dots,u_n \in [0,1]^n$, for a one-dimensional, continuous, strictly positive,
strictly decreasing and convex  function $W: [0,\infty) \to
(0,1]$ such that $W(0)=1$ and $W(\infty)=0$.

Note that $W$ can be extended to get a reliability function (by adding $W(u)=1$ for $u<0$). The inverse $W^{-1}$ of the function
$W$ is usually called the \emph{generator} of the Archimedean
 copula. As pointed out in  \cite{N06} (see pp. 116--117),   many
standard  copulas (such as the ones in Gumbel, Frank,
Clayton-Oakes and Ali-Mikhail-Haq families) are special cases of this
class. Note that Archimedean copulas are exchangeable (i.e., permutation invariant).

Random vectors of lifetimes having Archimedean survival copulas
are of great interest in reliability and actuarial sciences, but
also in many other applied contexts.  We refer
the reader to \cite{BS05,MS05,N06} and the references therein, for details, properties and applications of Archimedean survival
copulas.

Therefore, if $\textbf{X}=(X_1,\dots,X_n)$
has an strict Archimedean survival copula, then its joint reliability
function $\mathbf{\overline{F}}$ can be written in the form
\begin{equation}\label{survivalfunction}
\mathbf{\ov F}(x_1,\dots,x_n)=W(R_1(x_1)+\dots+R_n(x_n))
\end{equation}
for $n$ suitable continuous and strictly increasing functions
$R_i:[0,+\infty)\rightarrow [0,+\infty)$ given by $R_i(x)=W^{-1}(\ov F_i(x))$ such that
$R_i(0)=0$ and $R_i(\infty)=\infty$ for $i=1,\dots,n$ \cite[see][and the references therein for the details]{BS05}. 
As the Archimedean copulas are exchangeable, then $\textbf{X}$ is exchangeable if and only if (iff) the components are identically distributed (ID), that is, iff $\ov F_1=\dots =\ov F_n$.

This model is also known in the literature as the \emph{Time Transformed Exponential} (TTE) model and it is usually denoted as $\textup{TTE}(W,R_1,\dots,R_n)$ \cite[see, for instance,][]{MP10,MPR10a,MPR10b}. 
It is important to note that, in this model, we have 
\begin{align*}
	\overline{F}_i(x)&=\mathbf{\ov F}(0,\dots,0,x_i,0\dots,0)=W(R_i(x)),\\
	W^{-1}(x)&=R_i(\overline{F}^{-1}_i(x)),\\
	W(x)&=\overline{F}_i(R_i^{-1}(x))
\end{align*}
for $i=1,\dots,n$.
TTE models have been  considered in the  literature as an
appropriate way to describe bivariate lifetimes \cite[see][and the references therein]{BS05}. Their main
characteristic is that they ``separate'', in some sense, dependence
properties (based on $W$) and aging properties (based on $R$). TTE models
include relevant cases of dependent bivariate lifetimes, like
independent or Schur constant laws obtained when
$W(x)=\exp(- x)$ or $R_i(x)=x$, respectively, for $x\geq 0$ and $i=1,\dots,n$.

TTE models  can also be derived  from \emph{frailty models} (see Marshall and Olkin, 1988 or Oakes, 1989). In fact, in the frailty approach, it is assumed that $X_1,\dots,X_n$
are conditionally independent  given some positive (common) random environmental factor
$\Theta$, having conditional marginal reliability  functions $\ov
F_{i}(x|\theta) = \Pr(X_i > x \vert \Theta=\theta) = \ov
H^{\theta}_i(x)$ for some survival functions $\ov H_i$, $i=1,\dots,n$. Note that they satisfy a proportional hazard rate (Cox) model (PHR) with parameter $\theta$.  Thus, for this model, we have
\begin{align*}
\mathbf{\ov F}(x_1,\dots,x_n) & =E[\ov H^{\Theta}_i(x_1)\dots \ov H^{\Theta}_n(x_n)]\\
&=E[\exp(\Theta(\ln \ov H_1(x_1)))\dots\exp(\Theta(\ln \ov H_n(x_n)))]\\
&=E[\exp(-\Theta(-\ln \ov H_1(x_1) - \dots -\ln \ov H_n(x_n)))]\\
%&=W(-\ln\ov H_1(t_1)-\cdots-\ln \ov H_n(t_n)) 
&= W(R_1(x_1)+\dots+R_n(x_n)),% \ \ t_i\ge 0,
\end{align*}
where $W(x)=E[\exp(-\Theta x)]$ and $R_i(x)=-\ln\ov H_i(x)$, $i=1,\dots,n$. In this case, $W(x)$ is the Laplace transform (or  $W(-x)$ is the moment-generating function) associated to the distribution of $\Theta$. For example, in this model, the survival copula is of
Clayton-Oakes type iff the random parameter $\Theta$ has  a Gamma distribution, that is, when $W(x)=b^a(b+x)^{-a}$ for $x\geq 0$ and  $a,b>0$.

Throughout the paper, we say that a function $g$ is increasing (decreasing)  if $g(x)\leq g(y)$ ($\geq$) for all $x\leq y$. We say that $g>0$ is log-convex (log-concave) if $\ln g$ is convex (concave).

\section{Coherent systems under TTE models}

Let $T$ be the lifetime of a coherent (or semi-coherent) system with structure function $\psi$ and $n$ possibly dependent components having (nonnegative)
lifetimes $X_1,\dots, X_n$.  It is well known (see, e.g., Barlow and Proschan, 1975, p.\ 12) that $T$ can be written as $T=\psi(X_1,\dots,X_n)$ where
\begin{equation}\label{mps}
\psi(x_1,\dots,x_n)=\max_{i=1,\dots,r}\min_{j\in P_i}x_j
\end{equation}
and $P_1,\dots,P_r$ are the \emph{minimal path sets} of the system. A \textit{path set} is a set of indices $P$ such that the system works when all the components in $P$ work. A path set is \textit{minimal} if it does not contain other path sets. For example, the series system $X_P=\min_{j\in P}X_j$ has a unique minimal path set $P_1=P$ and the parallel system  $X^P=\max_{j\in P}X_j$ has $r=|P|$ minimal path sets given by $\{ i\}$ for $i\in P$.

From now on we assume that the random vector with the component lifetimes $\mathbf{X}=(X_1,\dots, X_n)$ satisfies the TTE model, that is, that  its joint reliability (survival) function can be written as in \eqref{survivalfunction} for some  strictly increasing functions
$R_i:[0,+\infty)\rightarrow [0,+\infty)$ such that
$R_i(0)=0$ and  $R_i(\infty)=\infty$, and a continuous, strictly positive,
strictly decreasing and convex function $W:[0,\infty) \to
(0,1]$ such that $W(0)=1$ and $W(\infty)=0$.  
Then the system reliability can be obtained as follows.

\begin{proposition}\label{mainresult}
Let $T=\psi(X_1,\dots, X_n)$ be the lifetime of a coherent system with minimal path sets $P_1,\dots, P_r$ and with
component lifetimes $(X_1,\dots, X_n)\sim \textup{TTE}(W,R_1,\dots,R_n)$.  Then the reliability function of  $T$ can be written as
\begin{equation}\label{RT}
\overline{F}_{T}(t)=\sum_{i=1}^{r} W\left(\sum_{k\in P_i}R_k(t)\right)-\sum_{i=1}^{r-1}\sum_{j=i+1}^{r} W\left(\sum_{k\in P_i\cup P_j}R_k(t)\right)+\ldots
+ (-1)^{r+1} W\left(\sum_{k\in P_1\cup\dots\cup P_r}R_k(t)\right) 
\end{equation}
and its probability density function as
\begin{align}\label{fT}
	f_{T}(t)&=
	 -\sum_{i=1}^{r} W'\left(\sum_{k\in P_i}R_k(t)\right)\sum_{k\in P_i}R'_k(t)+\sum_{i=1}^{r-1}\sum_{j=i+1}^{r} W'\left(\sum_{k\in P_i\cup P_j}R_k(t)\right)\sum_{k\in P_i\cup P_j}R'_k(t)+\dots\nonumber \\
	&\quad + (-1)^{r} W'\left(\sum_{k\in P_1\cup\dots\cup P_r}R_k(t)\right)\sum_{k\in P_1\cup\dots\cup P_r}R'_k(t)
	\end{align}
whenever $R_1,\dots,R_n$ and $W$ are differentiable.
\end{proposition}

\begin{proof}
From the representation based on the minimal path sets given in \eqref{mps} and the inclusion-exclusion formula, we have
\begin{align}
\overline{F}_{T}(t)&=\Pr\left(\max_{i=1,\dots,r} \min_{k\in P_i} X_k >t\right)\nonumber\\
&=\Pr\left(\bigcup_{i=1}^r \left\{ \min_{k\in P_i}X_k>t \right\} \right)\nonumber\\
&=\sum_{i=1}^{r} \Pr\left(\min_{k\in P_i}X_k>t\right)-\sum_{i=1}^{r-1}\sum_{j=i+1}^{r} \Pr\left(\min_{k\in P_i\cup P_j}X_k>t\right)+\dots\nonumber\\&\quad
 + (-1)^{r+1} \Pr\left(\min_{k\in P_1\cup\dots\cup P_r}X_k>t\right) .\label{FT}
\end{align}
Moreover,  note that, under the model defined  in \eqref{survivalfunction}, the reliability function of the series system $X_P=\min_{j\in P}X_j$ with components in $P$ can be obtained as 
\begin{equation}\label{RP}
\overline F_P(t)=\Pr\left(\min_{k\in P}X_k>t\right)=W\left(\sum_{k\in P}R_k(t)\right)
\end{equation}
since $R_i(0)=0$ for $i=1,\dots,n$. Hence, by using this expression in \eqref{FT}, we get \eqref{RT}. Differentiating \eqref{RT} we get \eqref{fT}.
\end{proof}

Let us see some simple examples.
\begin{itemize}
\item[(i)] Obviously, for a single component $X_i$, we have
\begin{equation*}
\overline F_i(t)= W(R_i(t))
\end{equation*}
as in  Section 2. The probability density function is $f_i(t)=-R'_i(t) W'(R_i(t))$ and the hazard rate $r_i(t)=f_i(t)/\overline F_i(t)=-R'_i(t) W'(R_i(t))/W(R_i(t))$ for $i=1,\dots,n$.
\item[(ii)] If we consider the series system  with lifetime $T = X_{1:n} = \min(X_1,\dots, X_n)$, then
\begin{equation}\label{seriesR}
\ov F_{1:n}(t) = \Pr(X_1 > t,\dots, X_n > t) =W(R_1(t)+\dots+R_n(t)).
\end{equation}
The  probability density function is
\begin{equation}\label{seriesf}f_{1:n}(t) = -(R'_1(t)+\dots +R'_n(t))W'(R_1(t)+\dots +R_n(t))
\end{equation}
and the hazard rate
\begin{equation}\label{seriesr}
	r_{1:n}(t)= -(R'_1(t)+\dots +R'_n(t))\frac{W'(R_1(t)+\dots +R_n(t))}{W(R_1(t)+\dots +R_n(t))}.
\end{equation}
The reliability function of a general series system with components in the set $P$ was given in \eqref{RP}. The density and hazard rate functions can be obtained from it.

\item[(iii)] For a parallel system with two components, the  minimal path sets are $P_1 =\{1\}$ and $P_2 = \{2\}$. Hence, the reliability function of the system lifetime is given by
\begin{align}
\overline F_{2:2}(t) &= \Pr(\max\{X_1,X_2\} > t)\nonumber\\
& =  \Pr(\{X_1 > t\} \cup \{X_2 > t\})\nonumber\\
& =  \Pr(X_1 > t)+\Pr(X_2 > t)-\Pr(X_1>t,X_2 > t)\nonumber\\
&=W(R_1(t))+W(R_2(t))-W(R_1(t)+R_2(t))\label{parallelR}.
\end{align}
The probability density function is
\begin{equation*}\label{parallelf}
	f_{2:2}(t) = (R'_1(t)+R'_2(t))W'(R_1(t)+R_2(t))-R'_1(t)W'(R_1(t))-R'_2(t)W'(R_2(t))
\end{equation*}	
and the hazard rate
\begin{align*}
	r_{2:2}(t) &= \frac{(R'_1(t)+R'_2(t))W'(R_1(t)+R_2(t))-R'_1(t)W'(R_1(t))
		-R'_2(t)W'(R_2(t))}{W(R_1(t))+W(R_2(t))-W(R_1(t)+R_2(t))}\nonumber\\
	&= \frac{R'_1(t)(W'(R_1(t)+R_2(t))-W'(R_1(t)))+R'_2(t)(W'(R_1(t)+R_2(t))-W'(R_2(t)))}{W(R_1(t))+W(R_2(t))-W(R_1(t)+R_2(t))}\label{parallelr}.
\end{align*}
The functions for a parallel system with $n$ components can be obtained in a similar way from \eqref{RT} and \eqref{fT}.

%\item[(iv)] Let us consider a parallel system with three components
%\item[(v)] Let us consider a $k$-out-of-$n$ system
\end{itemize}

Roughly speaking, Proposition \ref{mainresult} leads us to write the reliability and the probability density functions of any coherent system in terms of a suitable $n$-variate function depending on $W$ and $R_1,\dots,R_n$. It is important to note that when the components are ID, the expressions can be simplified. More specifically, we can state the following corollary.

\begin{corollary}\label{Corol-ID}
	If $T=\psi(X_1,\dots, X_n)$ is the lifetime of a coherent system with ID
component lifetimes $(X_1,\dots, X_n)\sim \textup{TTE}(W,R,\dots,R)$, then the reliability function of  $T$ is given by
\begin{equation}\label{RTid1}
\overline{F}_{T}(t)=\phi_{W}(R(t)),
\end{equation}
where $\phi_W:[0,\infty)\to[0,1]$ is a continuous decreasing function such that $\phi_W(0)=1$ and $\phi_W(\infty)=0$   given by
\begin{equation*}\label{RTid2}
\phi_{W}(x)=\sum_{i=1}^{r} W\left(|P_i|x\right)-\sum_{i=1}^{r-1}\sum_{j=i+1}^{r} W\left(|P_i\cup P_j|x\right)+\ldots
+ (-1)^{r+1} W\left(|P_1\cup\dots\cup P_r|x\right),
\end{equation*}
where $|P|$ represents the cardinal of the set $P$. Moreover, if $W$ and $R$ are differentiable, then the probability density function of the system is
\begin{equation*}\label{fTid1}
	f_{T}(t) = -R'(t)\phi'_{W}(R(t)).
	\end{equation*}
\end{corollary}

For example, for a series system with $n$ ID components,  $\phi_W(x)=W(n x)$, while for a parallel system with two ID components, $\phi_W (x)=2W(x)-W(2x)$.

\section{Stochastic comparisons}

Let us see now how the expressions obtained in the preceding section can be used to compare different coherent systems under  TTE models. We consider the following orders. Their basic properties can be seen in \cite{SS07}.

\begin{definition}\label{storders}
Let $X$ and $Y$ be two random variables. Then,
\begin{itemize}
\item[(i)] We say that $X$ is less than $Y$ in the (usual) stochastic (ST) order (shortly written
as $X \leq_{ST} Y$ or as $F_X \leq_{ST} F_Y$) if $F_X \geq F_Y$ .
\item[(ii)] We say that $X$ is less than $Y$ in the hazard rate (HR) order (shortly written as
$X \leq_{HR} Y$ or as $F_X \leq_{HR} F_Y$) if $\overline F_Y /\overline F_X$ is increasing.
\item[(iii)] We say that $X$ is less than $Y$ in the reversed hazard rate (RHR) order (shortly
written as $X \leq_{RHR} Y$ or as $F_X\leq_{RHR}F_Y$) if $F_Y /F_X$ is increasing.
\item[(iv)]  We say that $X$ is less than $Y$ in the likelihood ratio (LR) order (shortly
written as $X \leq_{LR} Y$ or as $F_X\leq_{LR}F_Y$) if $f_Y /f_X$ is increasing in the union of their supports, where $f_X$ and $f_Y$ are the probability density functions of $X$ and $Y$, respectively.
\end{itemize} 
\end{definition}

Some of these stochastic orders may be expressed in terms
of comparisons of the residual lifetimes or inactivity times of the corresponding random variables. Thus:
\begin{itemize}
	\item $X \le_{HR} Y$ iff
	$(X-t \vert \ X >t) \le_{ST} (Y-t \vert \ Y >t)$ for all $t$.
	\item $X \le_{RHR} Y$ iff
	$(t-X \vert \ X \le t) \ge_{ST} (t-Y \vert \ Y \le t)$ for all $t$.
	\item $X \le_{LR} Y$ iff $(X \vert \ a\leq X \leq b) \le_{ST}
	(Y \vert \ a\leq X \leq b)$ for all $a\leq b$.
%	\item $X \le_{MRL} Y$ if, and only if, $(X-t \vert \ X >t) \le_{ICX} (Y-t \vert \ Y >t)$ for all $t$.
\end{itemize}

The very well known relationships among these stochastic orders are summarized in 
Table \ref{tab:orders}. The reverse implications do not hold.

\begin{center}
	\begin{table}[h!]
		\begin{center}
			\begin{tabular}{ccccc}
				$X \leq_{LR} Y$ &  $\Rightarrow$ & $X \leq_{HR} Y$ &  %$\Rightarrow$ 
				& %$X \leq_{MRL} Y$  
				\\ 
				$\Downarrow$ & {} & $\Downarrow$ & & %$\Downarrow$ 
				\\
				$X \leq_{RHR} Y$ &  $\Rightarrow$ & $X \leq_{ST} Y$ &  $\Rightarrow$
				&  $E(X) \leq E(Y)$			
			\end{tabular}
		\end{center}
		\caption{Relationships among the main stochastic orders.}\label{tab:orders}
	\end{table}
\end{center}

\quad

Now let us study comparisons of systems under TTE models. First we give conditions to get orderings  between series systems with $n$ components when they  come from  TTE models with a common $W$ function (i.e. they have a common copula). Note that taking $n=1$ we obtain the conditions for the comparisons of the components. 

\begin{proposition}\label{series1}
Let $T_1=X_{1:n}$  and $T_2=Y_{1:n}$, where $(X_1,\dots,X_n)\sim\textup{TTE}(W,R_1,\dots,R_n)$ and  $(Y_1,\dots,Y_n)\sim\textup{TTE}(W,S_1,\dots,S_n)$.
\begin{itemize}
	\item[(i)] $T_1\leq_{ST} T_2$ if, and only if, $R_1+\dots+R_n\geq S_1+\dots+S_n$.
	\item[(ii)] $T_1\leq_{HR} T_2$ if $R_1+\dots+R_n\geq S_1+\dots+S_n$, $R'_1+\dots+R'_n\geq S'_1+\dots+S'_n$  and $W$ is log-concave.
	\item[(iii)]$T_1\leq_{RHR} T_2$ if $R_1+\dots+R_n\geq S_1+\dots+S_n$, {$R'_1+\dots+R'_n\leq S'_1+\dots+S'_n$  and  $\ov W$} is log-concave, where $\ov W=1-W$.
	\item[(iv)] $T_1\leq_{LR} T_2$ if $R_1+\dots+R_n\geq S_1+\dots+S_n$,  $R'_1+\dots+R'_n\geq S'_1+\dots+S'_n$, the function
	 $(S'_1+\dots+S'_n)/(R'_1+\dots+R'_n)$ is increasing and  $-W'$ log-concave. 
\end{itemize}
\end{proposition}

\begin{proof}

Let us define the functions $A(t)=R_1(t)+\dots+R_n(t)$ and $ B(t)=S_1(t)+\dots+S_n(t)$.
	
The proof of $(i)$ is immediate from \eqref{seriesR} (since $W$ is decreasing).

To prove $(ii)$ we consider the following function with the ratio of the respective reliability functions (obtained from \eqref{seriesR})
$$g(t)=\frac{W(S_1(t)+\dots+S_n(t))}{W(R_1(t)+\dots+R_n(t))}=\frac{W(B(t))}{W(A(t))}.$$
Then
$$g'(t)=_{sign}\frac{W'(B(t))}{W(B(t))}B'(t)-\frac{W'(A(t))}{W(A(t))}A'(t)$$
where $A(t)\geq B(t)$ (by the first assumption). Hence, if $W$ is log-concave, then $W'/W$ is decreasing and so $0\leq -W'(B(t))/W(B(t)) \leq -W'(A(t))/W(A(t))$. Finally, if $$B'(t)=S'_1(t)+\dots+S'_n(t)\leq R'_1(t)+\dots+R'_n(t)=A'(t),$$ we have $g'(t)\geq 0$ and so $g$ is increasing and  $T_1\leq_{HR} T_2$ holds.

The proof of $(iii)$ is similar to the preceding one.

From \eqref{seriesf}, the ordering in $(iv)$ $T_1\leq_{LR} T_2$ holds iff
\begin{equation}\label{LR}
\frac{B'(t)}{A'(t)}
\cdot
\frac{W'(B(t))}{W'\left(A(t)\right)}
\end{equation}
is increasing in  $t$. Note that we assume that the first fraction is increasing. If 
$$h(t):=
\frac{
	W'\left(B(t)\right)}{W'\left(A(t)\right)},$$
then 
$$h'(t)=_{sign}
\frac{
	W''\left(B(t)\right)}{W'\left(B(t)\right)}B'(t)-\frac{
	W''\left(A(t)\right)}{W'\left(A(t)\right)}A'(t).$$
Here we know that $W$ is convex and decreasing and so $W''\geq 0$ and $W'\leq 0$. Hence, if $-W'$ is log-concave, then $-W''/W'$ is increasing and we have 
$$-\frac{
	W''\left(A(t)\right)}{W'\left(A(t)\right)}\geq -\frac{
	W''\left(B(t)\right)}{W'\left(B(t)\right)}\geq 0$$
since we assume $A=R_1+\dots+R_n\geq S_1+\dots+S_n=B$. 
Finally, as we also assume that 
$A'\geq B'\geq 0$, 
we get that $h'\geq 0$ and  $h$ is increasing  in $t$, obtaining the LR ordering. 
\end{proof}

Some log-concave $W$ functions are  $W(x)=(1-\theta)/(e^x-\theta)$,  $\theta\in[-1,0)$,  $W(x)=-\ln(1-e^{-x}+e^{-\theta-x})/\theta$, $\theta\in(-\infty,0]$, and  $W(x)=\exp((1-e^x)/\theta)$, $\theta\in(0,1]$, which correspond to   Ali-Mikhail-Haq, Frank and Gumbel-Barnett copulas, respectively, \cite[see lines  3, 5 and 9 of Table 4.1 in][pp. 116--119]{N06}.

Note that $W$ is log-concave iff the hazard (failure) rate function associated to $W$  given by $-W'(t)/W(t)$ is increasing, that is, the reliability function $W$ is increasing failure (or hazard) rate (IFR). Also note that $\ov W$ is log-concave iff $W$ is decreasing reversed failure rate (DRFR) and  that  $-W'$ is log-concave iff $W$ is increasing likelihood ratio (ILR). It is well known that ILR implies both IFR and DRFR.

Next we study the series systems obtained from  TTE models with different $W$ functions and common $R_1,\dots,R_n$ functions.

\begin{proposition}\label{series2}
Let $T_1=X_{1:n}$  and $T_2=Y_{1:n}$, where  $(X_1,\dots,X_n)\sim\textup{TTE}(W_1,R_1,\dots,R_n)$ and  $(Y_1,\dots,Y_n)\sim\textup{TTE}(W_2,R_1,\dots,R_n)$. Then,
\begin{itemize}
\item[(i)] $T_1\leq_{ST} T_2$ for all $R_1,\dots,R_n$ iff $W_1\leq W_2$.
\item[(ii)] $T_1\leq_{HR} T_2$ for all $R_1,\dots,R_n$ iff $W_2/W_1$ is increasing.
\item[(iii)] $T_1\leq_{RHR} T_2$ for all $R_1,\dots,R_n$ iff $\overline W_2/\overline W_1$ is increasing, where $\overline W_i=1-W_i$, $i=1,2$.
\item[(iv)] $T_1\leq_{LR} T_2$ for all $R_1,\dots,R_n$ iff $W_2^\prime/W_1^\prime$ is increasing. 
\end{itemize}
\end{proposition}

\begin{proof}
 It follows directly from \eqref{seriesR}, \eqref{seriesf}, \eqref{seriesr} and Definition \ref{storders}.
\end{proof}

Of couse, it is possible to provide conditions for the comparison of $T_1=X_{1:n}$  and $T_2=Y_{1:n}$, where  $(X_1,\dots,X_n)\sim\textup{TTE}(W_1,R_1,\dots,R_n)$ and  $(Y_1,\dots,Y_n)\sim\textup{TTE}(W_2,S_1,\dots,S_n)$ just by considering  $T_3=Z_{1:n}$ where  $(Z_1,\dots,Z_n)\sim\textup{TTE}(W_1,S_1,\dots,S_1)$ and using jointly Propositions \ref{series1} and \ref{series2}.

Next, we study comparisons of  parallel systems with two components. First, we consider the case of a common $W$ function and a common marginal.

\begin{proposition}
		Let $T_1=X_{2:2}$  and $T_2=Y_{2:2}$, where  $(X_1,X_2)\sim\textup{TTE}(W,R_1,R_2)$ and  $(Y_1,Y_2)\sim\textup{TTE}(W,S_1,S_2)$. If $R_1=S_1$ and $R_2\geq S_2$, then  $T_1\leq_{ST} T_2$.
	\end{proposition}
\begin{proof}
		We note that $T_1\leq_{ST}T_2$ is equivalent to
		\[
		W(R_2(t))-W(R_1(t)+R_2(t))\leq W(S_2(t))-W(R_1(t)+S_2(t)),
		\]
		which holds if $R_2\geq S_2$ since $W$ is decreasing and convex.
	\end{proof}

In the following proposition we study parallel systems with two ID components and a common $W$ function.

\begin{proposition}
Let $T_1=X_{2:2}$  and $T_2=Y_{2:2}$, where  $(X_1,X_2)\sim\textup{TTE}(W,R,R)$ and  $(Y_1,Y_2)\sim\textup{TTE}(W,S,S)$. Then,
\begin{itemize}
\item[(i)] $T_1\leq_{ST} T_2$ iff $R\geq S$.
\item[(ii)] $T_1\leq_{HR} T_2$ if $R\geq S$, $R'\geq S'$ and $\frac{W'(2x)-W'(x)}{2W(x)-W(2x)}$ is decreasing.
\item[(iii)] $T_1\leq_{RHR} T_2$ if $R\geq S$, $R'\leq S'$ and $\frac{W'(2x)-W'(x)}{1-2W(x)+W(2x)}$ is decreasing. 
\item[(iv)] $T_1\leq_{LR} T_2$ if $R\geq S$, $R'\geq S'$, $S'/R'$ is increasing and $\frac{W''(x)-2W''(2x)}{W'(2x)-W'(x)}$ is increasing and nonnegative.
\end{itemize}
\end{proposition}

\begin{proof}
From \eqref{parallelR} and \eqref{RTid1}, the system reliability functions can be written as $\overline{F}_{T_1}(t)=\phi_{W}(R(t))$ and $\overline{F}_{T_2}(t)=\phi_{W}(S(t))$, 
where
\begin{equation*}
\phi_{W}(x)=2W(x)-W(2x)
\end{equation*}
is a decreasing function such that 
$\phi_{W}(0)=1$ and $\phi_{W}(\infty)=0$.

Hence, the ordering in $(i)$ holds if and only if $R\geq S$.

To prove $(ii)$ we note that $T_1\leq_{HR} T_2$ holds iff
\[
\frac{\phi_{W}(S(t))}{\phi_{W}(R(t))}
\]
is increasing or, equivalently, iff
\[
-\frac{\phi_{W}'(R(t))}{\phi_{W}(R(t))}R'(t)\geq -\frac{\phi_{W}'(S(t))}{\phi_{W}(S(t))}S'(t).
\]
Then, the assertion follows since $R\geq S$, $R'\geq S'\geq 0$,  and 
\[
-\frac{\phi_{W}'(x)}{\phi_{W}(x)}=\frac{2W'(2x)-2W'(x)}{2W(x)-W(2x)}
\]
is nonnegative and decreasing.

%The proof of $(iii)$ is similar to the preceding one.

%\quad

To prove $(iii)$ we note that $T_1\leq_{RHR} T_2$ holds iff
\[
\frac{\bar \phi_{W}(S(t))}{\bar \phi_{W}(R(t))}
\]
	is increasing,  where $\bar\phi_W(x)=1-\phi_W(x)$ is a nonnegative increasing function such that $\bar\phi_W(0)=0$ and $\bar\phi_W(\infty)=1$. Hence this condition is equivalent to 
	\[\frac{\bar \phi_{W}'(S(t))}{\bar \phi_{W}(S(t))}S'(t)\geq 
	\frac{\bar \phi_{W}'(R(t))}{\bar \phi_{W}(R(t))}R'(t)\geq 0.
	\]
	Then, the assertion follows since $R\geq S$, $S'\geq R'\geq 0$,  and 
	\[
	\frac{\bar \phi_{W}'(x)}{\bar \phi_{W}(x)}=\frac{2W'(2x)-2W'(x)}{1-2W(x)+W(2x)}
	\]
	is nonnegative and decreasing.

To prove $(iv)$, we use that  $T_1\leq_{LR} T_2$ holds iff
\[\frac{S'(t)}{R'(t)}\cdot
\frac{\phi_{W}'(S(t))}{\phi_{W}'(R(t))}
\]
is increasing. This condition holds if $S'/R'$ is increasing and $\phi_{W}'(S(t))/\phi_{W}'(R(t))$ is increasing (note that both are nonnegative). Note that we assume that $S'/R'$ is increasing. Hence, we conclude the proof taking into account that {$g(t):=\phi_{W}'(S(t))/\phi_{W}'(R(t))$ is increasing in $t$ iff
$$g'(t)=_{sign} \frac{S'(t)\phi_{W}''(S(t))}{\phi_{W}'(S(t))}-\frac{R'(t)\phi_{W}''(R(t))}{\phi_{W}'(R(t))}\geq 0.
$$
Here we know that $R'\geq S'\geq 0$. Moreover, we assume that  
$$-\frac{\phi_{W}''(x)}{\phi_{W}'(x)}=
\frac{W''(x)-2W''(2x)}{W'(2x)-W'(x)}$$ is increasing and nonnegative. Hence, as $R\geq S$, we get that $g$ is increasing and the proof is completed.}
\end{proof}

Note that when $W(x)=\exp\left(-x^{1/2}\right)$ for $x\geq 0$, which corresponds to a Gumbel-Hougaard copula \cite[see line 4 in Table 4.1 of][p. 116]{N06}, we get that $\frac{W'(2x)-W'(x)}{2W(x)-W(2x)}$ is decreasing and so we can apply the condition in $(ii)$ of the preceding proposition. However,  it cannot be applied for the Clayton-Oakes copula \cite[see line 1 in Table 4.1 of][p. 116]{N06} with $\theta=1$ since  we get  $W(x)=(1+x)^{-1}$ and  so $\frac{W'(2x)-W'(x)}{2W(x)-W(2x)}$ is not  decreasing. 
%In a similar way, if we study the function of condition $(iv)$ (LR ordering), we see that $\frac{W''(x)-2W''(2x)}{W'(2x)-W'(x)}$ is decreasing and nonnegative for the above Gumbel-Hougaard copula. %\red{So we cannot apply $(iv)$ in the preceding proposition to this copula. The same happen for the above Clayton-Oakes copula since this function is nonmonotonic and take negative values. }

In the following proposition we study parallel systems with two components having a common $R$ function.

\begin{proposition}
	Let $T_1=X_{2:2}$  and $T_2=Y_{2:2}$, where  $(X_1,X_2)\sim\textup{TTE}(W_1,R,R)$ and  $(Y_1,Y_2)\sim\textup{TTE}(W_2,R,R)$. Let $\phi_i(x):=2W_i(x)-W_i(2x)$ for $i=1,2$.
	Then,
	\begin{itemize}
		\item[(i)] $T_1\leq_{ST} T_2$ for all $R$ iff $\phi_1 \leq \phi_2$.
		\item[(ii)] $T_1\leq_{HR} T_2$ for all $R$ iff  $\phi_2/\phi_1$ is increasing.
		\item[(iii)] $T_1\leq_{RHR} T_2$ for all $R$ iff  $(1-\phi_2)/(1-\phi_1)$ is increasing.
		\item[(iv)] $T_1\leq_{LR} T_2$ for all $R$ iff  $\phi'_2/\phi'_1$ is increasing.
	\end{itemize}
\end{proposition}

\begin{proof}
	From \eqref{parallelR} and \eqref{RTid1}, the system reliability functions can be written as $\overline{F}_{T_1}(t)=\phi_{1}(R(t))$ and $\overline{F}_{T_2}(t)=\phi_{2}(R(t))$, 
	where 
	$\phi_1$ and $\phi_2$ are decreasing functions such that 
	$\phi_{i}(0)=1$ and $\phi_{i}(\infty)=0$ for $i=1,2$.
	
	Then $T_1\leq_{ST} T_2$ holds for all $R$ iff $\phi_1(R(t))\leq \phi_2(R(t))$ which is equivalent to $\phi_1\leq \phi_2$. 
	
	Analogously, $T_1\leq_{HR} T_2$ holds for all $R$ iff $\phi_2(R(t))/\phi_1(R(t))$ is increasing in $t$ for all $R$. As $R$ is increasing, this condition is equivalent to  $\phi_2/ \phi_1$ is increasing.
	
	The proofs of $(iii)$ and $(iv)$ are similar.
\end{proof}

For example, if $W_i(x)=(1+x)^{-\theta_i}$, for $x\geq 0$, $\theta_i>0$ and $i=1,2$, which correspond to Clayton-Oakes copulas, we have that $\phi_1\leq\phi_2$ iff  $\theta_1\geq \theta_2$. Under this condition, it also holds that $\phi_2/\phi_1$, $(1-\phi_2)/(1-\phi_1)$ and $\phi'_2/\phi'_1$ are increasing. Hence all the orderings in the preceding proposition holds.

The preceding proposition can be extended to general coherent systems as follows.

\begin{proposition}\label{Pmain}
	Let $T_1=\psi_1(X_1,\dots,X_n)$  and $T_2=\psi_2(Y_1,\dots,Y_n)$ be the lifetimes of two coherent systems whose component lifetimes satisfy $(X_1,\dots,X_n)\sim\textup{TTE}(W_1,R,\dots,R)$ and  $(Y_1,\dots, Y_n)\sim\textup{TTE}(W_2,R,\dots, R)$. Let $\phi_i$ be the functions obtained from \eqref{RTid1} for the systems $T_i$ for $i=1,2$. Then,
	\begin{itemize}
		\item[(i)] $T_1\leq_{ST} T_2$ for all $R$ iff $\phi_1 \leq \phi_2$.
		\item[(ii)] $T_1\leq_{HR} T_2$ for all $R$ iff  $\phi_2/\phi_1$ is increasing.
		\item[(iii)] $T_1\leq_{RHR} T_2$ for all $R$ iff  $(1-\phi_2)/(1-\phi_1)$ is increasing.
		\item[(iv)] $T_1\leq_{LR} T_2$ for all $R$ iff  $\phi'_2/\phi'_1$ is increasing.
	\end{itemize}
\end{proposition}

\begin{proof}
	From \eqref{RTid1}, the reliability function of $T_i$ can be written as 
\begin{equation}
\overline{F}_{T_i}(t)=\phi_{i}(R(t)),
\end{equation}
where $\phi_{i}$ are decreasing functions such that 
$\phi_{i}(0)=1$ and $\phi_{i}(\infty)=0$ for $i=1,2$.

Then $T_1\leq_{ST} T_2$ holds for all $R$ iff $\phi_1(R(t))\leq \phi_2(R(t))$ which is equivalent to $\phi_1\leq \phi_2$. 

Analogously, $T_1\leq_{HR} T_2$ holds for all $R$ iff $\phi_2(R(t))/\phi_1(R(t))$ is increasing in $t$ for all $R$. As $R$ is increasing, this condition is equivalent to  $\phi_2/ \phi_1$ is increasing.

The proofs of $(iii)$ and $(iv)$ are similar.
\end{proof}

If the coherent systems have a common structure and a common $W$ function but different $R$ functions then we obtain the following proposition.

\begin{proposition}
	Let $T_1=\psi(X_1,\dots,X_n)$  and $T_2=\psi(Y_1,\dots,Y_n)$ be the lifetimes of two coherent systems whose component lifetimes satisfy $(X_1,\dots,X_n)\sim\textup{TTE}(W,R,\dots,R)$ and  $(Y_1,\dots, Y_n)\sim\textup{TTE}(W,S,\dots, S)$. Let $\phi_W$ be the (common) function obtained from \eqref{RTid1} for the systems $T_i$ for $i=1,2$. Then,
	\begin{itemize}
		\item[(i)] $T_1\leq_{ST} T_2$ if and only if $R\geq S$.
		\item[(ii)] $T_1\leq_{HR} T_2$ if $R\geq S$, $R'\geq S'$ and  $-\phi'_W/\phi_W$ is increasing.
		\item[(iii)] $T_1\leq_{RHR} T_2$ if $R\geq S$, $R'\leq S'$ and $-\phi'_W/(1-\phi_W)$ is {decreasing}.
\item[(iv)] $T_1\leq_{LR} T_2$ if $R\geq S$, $R'\geq S'$, $S'/R'$ is increasing and {$-\phi''_W/\phi'_W$ is nonnegative and  increasing}.
	\end{itemize}
\end{proposition}

\begin{proof}
	From \eqref{RTid1}, the reliability functions of $T_1$ and $T_2$ can be written as 
	$\overline{F}_{T_1}(t)=\phi_W(R(t))$ and 
	$\overline{F}_{T_2}(t)=\phi_W(S(t))$,
	where $\phi$ is a nonnegative decreasing function such that 
	$\phi_W(0)=1$ and $\phi_W(\infty)=0$. Note that this function depends on $W$ and on the structure function $\psi$. So we have a common function $\phi_W$ for both systems.

Then $T_1\leq_{ST} T_2$ holds iff $\phi_W(R(t))\leq \phi_W(S(t))$ for all $t$ which is equivalent to $R\geq S$ (since $\phi_W$ is decreasing). 

Analogously, $T_1\leq_{HR} T_2$ holds iff $g(t):=\phi_W(S(t))/\phi_W(R(t))$ is increasing in $t$. Differentiating we get
$$g'(t)=_{sign} \frac{\phi'_W(S(t))}{\phi_W(S(t))}S'(t)- \frac{\phi'_W(R(t))}{\phi_W(R(t))}R'(t).$$
Hence, if we assume that $R\geq S$ and $-\phi'_W/\phi_W$ is increasing, we get 
$$- \frac{\phi'_W(R(t))}{\phi_W(R(t))}\geq -\frac{\phi'_W(S(t))}{\phi_W(S(t))}\geq 0.$$
Finally, if $R'\geq S'\geq 0$, we have $g'\geq 0$ and the HR ordering holds.

To prove $(iii)$ we note that  $T_1\leq_{RHR} T_2$ holds iff $$g(t):=\frac {F_{T_2}(t)}{F_{T_1}(t)}=\frac{1-\phi_W(S(t))}{1-\phi_W(R(t))}$$ is increasing in $t$. Differentiating we get
$$g'(t)=_{sign} -\frac{\phi'_W(S(t))}{1-\phi_W(S(t))}S'(t)+ \frac{\phi'_W(R(t))}{1-\phi_W(R(t))}R'(t).$$
Hence, if we assume that $R\geq S$ and $-\phi'_W/(1-\phi_W)$ is {decreasing}, we get 
$$- \frac{\phi'_W(S(t))}{1-\phi_W(S(t))}\geq -\frac{\phi'_W(R(t))}{1-\phi_W(R(t))}\geq 0.$$
Finally, if $S'\geq R'\geq 0$, we have $g'\geq 0$ and the RHR ordering holds.

To prove $(iv)$ we note that  $T_1\leq_{LR} T_2$ holds iff $$\frac {f_{T_2}(t)}{f_{T_1}(t)}= \frac {S'(t)}{R'(t)}\cdot\frac{\phi'_W(S(t))}{\phi'_W(R(t))}$$ is increasing in $t$. 
Here we assume that $S'/R'$ is increasing. Then, we define 
$$h(t):=\frac{\phi'_W(S(t))}{\phi'_W(R(t))}$$
and differentiating we get
$$h'(t)=_{sign} \frac{\phi''_W(S(t))}{\phi'_W(S(t))}S'(t)- \frac{\phi''_W(R(t))}{\phi'_W(R(t))}R'(t).$$
Hence, if we assume that $R\geq S$ and $-\phi''_W/\phi'_W$ is nonnegative and  increasing, we get 
$$- \frac{\phi''_W(R(t))}{\phi'_W(R(t))}\geq -\frac{\phi''_W(S(t))}{\phi'_W(S(t))}\geq 0.$$
Finally, if $R'\geq S'\geq 0$, we have $h'\geq 0$ and the LR ordering holds. 
\end{proof}

By combining the two preceding propositions we obtain the following corollary.
\begin{corollary}
Let $T_1=\psi(X_1,\dots,X_n)$  and $T_2=\psi(Y_1,\dots,Y_n)$ be the lifetimes of two coherent systems whose component lifetimes satisfy  $(X_1,\dots,X_n)\sim\textup{TTE}(W_1,R,\dots,R)$ and  $(Y_1,\dots, Y_n)\sim\textup{TTE}(W_2,S,\dots, S)$. Let $\phi_i$ be the functions obtained from \eqref{RTid1} for the systems $T_i$ for $i=1,2$. Then,
		\begin{itemize}
			\item[(i)] $T_1\leq_{ST} T_2$ if and only if $R\geq S$ and $\phi_1\leq\phi_2$.
\item[(ii)] $T_1\leq_{HR} T_2$ if $R\geq S$, $R'\geq S'$, and $\phi_2/\phi_1$ and { $-\phi'_1/\phi_1$  (or $-\phi'_2/\phi_2$)} are increasing.
\item[(iii)] $T_1\leq_{RHR} T_2$ if $R\geq S$, $R'\leq S'$, $(1-\phi_2)/(1-\phi_1)$ {is increasing and $-\phi'_1/(1-\phi_1)$ (or $-\phi'_2/(1-\phi_2)$) is decreasing.}
\item[(iv)] $T_1\leq_{LR} T_2$ if $R\geq S$, $R'\geq S'$ and $S'/R'$ is increasing, and $\phi'_2/\phi'_1$ is increasing {and $-\phi''_1/\phi'_1$ (or $-\phi''_2/\phi'_2$) is nonnegative and  increasing}.
\end{itemize}
	\end{corollary}

%AQUI

Finally, we study some particular cases of special interest. In the following proposition we compare a  single component with a series system with $n$ components. Note that we know that $X_{1:n}\leq_{ST}X_i$ for all $i=1,\dots,n$ and for any model (including TTE models). However, we need some conditions for the other orders given in the following propositon.

\begin{proposition}
	Let $T_1=X_{1:n}$  and $T_2=X_{i}$ for an $i\in\{1,\dots,n\}$, where  $(X_1,\dots,X_n)\sim\textup{TTE}(W,R_1,\dots,R_n)$. Then,
	\begin{itemize}
		\item[(i)] $T_1\leq_{HR} T_2$ if $W$ is log-concave.
		\item[(ii)] $T_1\leq_{LR} T_2$ if  $-W'$ is log-concave and   
		\begin{equation}\label{condLR}
		\frac{R''_i(t)}{R'_i(t)}\geq \frac 1{n-1}\sum_{j\neq i}\frac{R''_j(t)}{R'_j(t)}\text{ for all }t\geq 0.
		\end{equation}
	\end{itemize}
\end{proposition}

\begin{proof}
To prove (i) we note that,  from \eqref{seriesR}, $T_1\leq_{HR} T_2$ holds iff
$$g(t):=\frac{W(R_i(t))}{W(\sum_{j=1}^nR_i(t))}$$
is increasing in $t$. Differentiating we have
\begin{align*}
g'(t)&=_{sign}\frac{W'(R_i(t))}{W(R_i(t))}R'_i(t)-
\frac
{W'(\sum_{j=1}^nR_j(t))}{W(\sum_{j=1}^nR_j(t))}\left(\sum_{j=1}^nR'_j(t)\right), 
\end{align*}
where $R'_j\geq 0$ for  $j=1,\dots,n$. 
Hence, if $W$ is log-concave, then $-W'/W$ is nonnegative and increasing. So, $g'$ is non-negative (since $R_i\leq \sum_{j=1}^nR_j$) and we prove the stated result.

\quad

Analogously, to prove $(ii)$,  we use that, from \eqref{seriesf}, $T_1\leq_{LR} T_2$ holds iff
$$\frac{W'(R_i(t))}{W'(\sum_{j=1}^nR_j(t))}\cdot \frac{R'_i(t)}{\sum_{j=1}^nR'_j(t)}$$
is increasing in $t$.
The first fraction
$$\alpha(t):=\frac{W'(R_i(t))}{W'(\sum_{j=1}^nR_j(t))}$$
 is increasing iff
$$\alpha'(t)=\frac{W''(R_i(t))}{W'(R_i(t))}R_i'(t)-\frac {W''(\sum_{j=1}^nR_j(t))}{W'(\sum_{j=1}^nR_j(t))}\left(\sum_{j=1}^nR'_j(t)\right)\geq 0.$$
As $W$ is decreasing and  convex, $W'\leq 0$ and $W''\geq 0$.
Moreover  $R_1,\dots,R_n$ are nonnegative and increasing.  Then, if $-W'$ is log-concave, then $-W''/W'$ is nonnegative and increasing and so
$$    -\frac {W''(\sum_{j=1}^nR_j(t))}{W'(\sum_{j=1}^nR_j(t))}\left(\sum_{j=1}^nR'_j(t)\right)\geq \frac{W''(R_i(t))}{W'(R_i(t))}R_i'(t) \geq 0$$
for all $t\geq 0$ and $\alpha$ is increasing. 

The second fraction is
$$\beta(t):=\frac{R'_i(t)}{\sum_{j=1}^nR'_j(t)}\geq 0.$$
Differentiating we have
$$\beta'(t)=_{sign}R''_i(t)\sum_{j=1}^nR'_j(t)-R'_i(t)\sum_{j=1}^nR''_i(t)=\sum_{j\neq i}[R''_i(t)R'_j(t)-R''_j(t)R'_i(t)]\geq 0$$
when \eqref{condLR} holds. Hence both fractions are nonegative and  increasing and we have the LR-ordering.
\end{proof}
%%%%AQUI

Again note that $W$ is log-concave iff the reliability function $W$ is IFR.
Also note that \eqref{condLR} holds when $R_1=\dots=R_n$.
Note that if we consider the log-convex function $W(x)=(1+x)^{-1}$ (Clayton-Oakes copula) and the functions
	\[
	R_1(t)=e^{2 t}-1\]
and
\[R_2(t)=e^{t}-1 \]
for $t\geq 0$,
then $X_{1:2}\not\leq_{HR} X_{1}$. However, it holds $X_{1:2}\leq_{HR} X_{2}$. In Figure \ref{HazardratesProp49a} (left) we can see the corresponding hazard rates for the log-convex function $W(x)=(1+x)^{-1}$, while in Figure \ref{HazardratesProp49a} (right) we show them for the log-concave function $W(x)=\exp(2(1-e^x))$ (which corresponds to the Gumbel-Barnett copula). In this last case we have $X_{1:2}\leq_{HR} X_1\leq_{HR} X_2$.
	
\begin{figure}
\includegraphics[width=\textwidth]{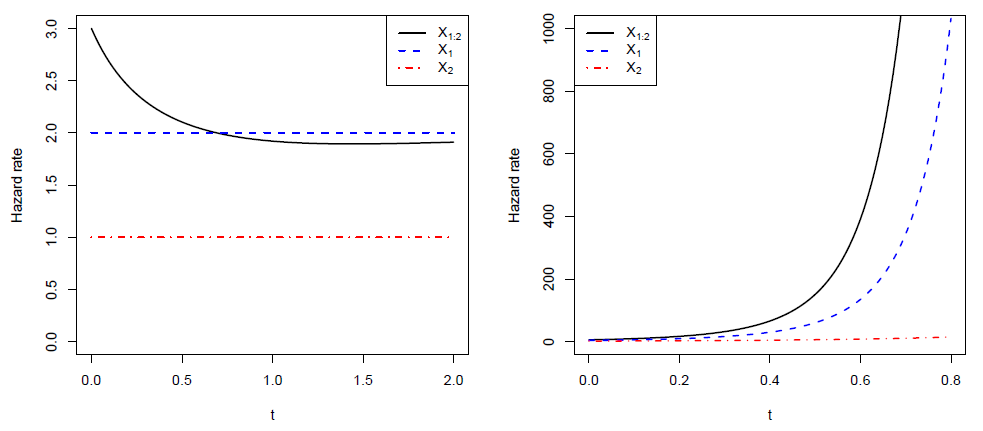}%{Prop49a.pdf}
\caption{Hazard rate functions of $X_{1:2}$, $X_1$ and $X_2$ for $R_1(t)=e^{2 t}-1$ and $R_2(t)=e^{t}-1$, $W(t)=(1+t)^{-1}$  (left) and $W(t)=\exp((1-e^t)/0.5)$ (right) for $t\geq 0$.}\label{HazardratesProp49a}
\end{figure}

%%%%%%%%%%

In the following proposition we compare series and parallel systems with two components satisfying the same TTE model. Note that  $X_{1:2}\leq_{ST}X_{2:2}$ holds for any model.

\begin{proposition}
	Let $T_1=X_{1:2}$  and $T_2=X_{2:2}$ where  $(X_1,X_2)\sim\textup{TTE}(W,R_1,R_2)$. Then,
	$T_1\leq_{HR} T_2$ for all $R_1,R_2$ if  $W$ is log-concave.
\end{proposition}

\begin{proof}
The HR ordering $T_1\leq_{HR} T_2$ holds iff 
\[
\frac{ W(R_1(t))+W(R_2(t))-W(R_1(t)+R_2(t))}{W(R_1(t)+R_2(t))}
\]
is increasing, that is, iff the function 
\[
 g(t):=\frac{W(R_1(t))+W(R_2(t))}{W(R_1(t)+R_2(t))}
\]
is increasing. Differentiating, we get
\begin{align*}
g'(t)&=_{sign}
[W'(R_1(t))R'_1(t)+W'(R_2(t))R'_2(t)]
W(R_1(t)+R_2(t))\\
&\quad -[W(R_1(t))+W(R_2(t))] W'(R_1(t)+R_2(t))[R'_1(t)+R'_2(t)]\\
&=R'_1(t)[W'(R_1(t))W(R_1(t)+R_2(t))-[W(R_1(t))+W(R_2(t))] W'(R_1(t)+R_2(t))]\\
&\quad +R'_2(t)[W'(R_2(t))W(R_1(t)+R_2(t))-[W(R_1(t))+W(R_2(t))] W'(R_1(t)+R_2(t))].
\end{align*}
If $W$ is  log-concave (i.e. $W'/W$ is decreasing), we have 
$$\frac{W'(R_i(t))}{W(R_i(t))}\geq \frac{W'(R_1(t)+R_2(t))}{W(R_1(t)+R_2(t))}$$
for $i=1,2$ and for all $t$.
Moreover, as $W$ is non-negative and decreasing, we have
$$-W(R_i(t))W'(R_1(t)+R_2(t))\geq 0$$
for $i=1,2$ and for all $t$.
Therefore 
$g'(t)\geq 0$ for all $t$ and the proof is completed.
\end{proof}

Again the condition for the HR ordering in the preceding proposition holds when $W$ is IFR. For example, it holds when the components are independent, that is, when $W(x)=e^{-x}$ for $x\geq 0$ (a well known property). However, it does not hold in the Frailty model when $\Theta$ has a Gamma distribution (i.e. when $K$ is a Clayton-Oakes copula) since $W(x)=b^a/(b+x)^a$ is strictly DFR (Pareto reliability). In this case the series and parallel systems are not always HR ordered. For example, if we consider $W(x)=(1+3x)^{-1/3}$, $R_1(x)=(e^x-1)/10$ and $R_2(x)=x$ for $x\geq 0$, the previous proposition cannot be applied and, as can be seen in Figure \ref{HazardratesProp412}, the HR ordering does not hold.

\begin{figure}
	\centering
	\includegraphics[width=0.5\textwidth]{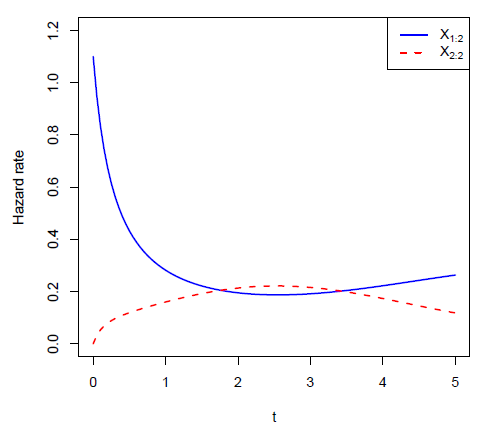}%{Prop412.pdf}
	\caption{Hazard rates of $X_{1:2}$ and $X_{2:2}$ for $W(x)=(1+3t)^{-1/3}$, $R_1(t)=(e^t-1)/10$ and $R_2(t)=t$ for $t\geq 0$.}\label{HazardratesProp412}
\end{figure}

As a particular case of Proposition \ref{Pmain} we obtain the following result to compare series and parallel systems from the same model and a common $R$ function.

\begin{proposition}
	Let $T_1=X_{1:2}$  and $T_2=X_{2:2}$ where  $(X_1,X_2)\sim\textup{TTE}(W,R,R)$. Then,
\begin{itemize}
\item[(i)] $T_1\leq_{ST} T_2$ for all $R$ and all $W$.  
\item[(ii)] $T_1\leq_{HR} T_2$  for all $R$ iff  $W(x)/W(2x)$ is increasing.
\item[(iii)] $T_1\leq_{RHR} T_2$  for all $R$ iff $\ov W(x)/\ov W(2x)$ is increasing, where $\ov W=1-W$.
\item[(iv)] $T_1\leq_{LR} T_2$  for all $R$ iff $W'(x)/W'(2x)$ is increasing.
\end{itemize}
\end{proposition}

\begin{proof}
From \eqref{seriesR} and \eqref{parallelR}, we have 
$\phi_1(x)=W(2x)$ and $\phi_2(x)=2W(x)-W(2x)$. Then the proof follows from Proposition \ref{Pmain}.
\end{proof}

Note that  $W(x)/W(2x)$ is increasing if $W$ is log-concave, but the reverse is not true.

Similar conditions can be given to compare two other coherent systems with specific structures and components satisfying TTE models with the same $R$ function. Let us see an example.

\begin{example}\label{Exairplane}
Let us consider an aircraft with four engines, two per wing. The aircraft can fly when, at least, an engine works on each wing. If $X_1$, $X_2$, $X_3$, $X_4$ are the lifetimes of the engines, then the lifetime of the engine's system in this aircraft is
\[
T = \min\{\max(X_1, X_2), \max(X_3, X_4)\}.
\]
We assume that the component lifetimes are identically distributed. However, given that the components share the same environment, let us assume that their lifetimes are dependent and that this dependence is modelled by a Time Transformed Exponential model $TTE(W,R,R,R,R)$.  Then the system reliability is
\begin{align*}
	\Pr(T>t)&=\Pr(\min\{\max(X_1, X_2), \max(X_3, X_4)\}>t)\\
&= 4W(2R(t))-4W(3R(t))+W(4R(t))\\
	&= \phi(R(t)),
\end{align*}
where 
\begin{equation}\label{airplane}
\phi(x)=4W(2x)-4W(3x)+W(4x).
\end{equation}
An interesting situation is when the TTE model is derived from a frailty model, where $W(x)=E[\exp(-\Theta x)]$ for an environmental random variable $\Theta$, i.e., $W$ is the Laplace transform of the so-called frailty $\Theta$. Many distributions can be chosen for the frailty distribution. One of the most common ones is the gamma distribution whose Laplace transform is given by
\[
W(x)=\frac{b^{a}}{(b+x)^{a}},\text{ for all }x\geq 0,
\]
for $a,b>0$. 
Let us consider now two different dependence structures for the component lifetimes given by two time transformed exponential models $TTE(W_1,R,R,R,R)$ and $TTE(W_2,R,R,R,R)$, where $W_i$, $i=1,2$, is given by the Laplace transform of two different gamma distributions. Let, for example, $W_1(x)=1/(1+x)$ ($a_1=1,b_1=1$) and $W_2(x)=3/(3+x)$ ($a_2=1,b_2=3$) and let $\phi_1$ and $\phi_2$ be the corresponding functions obtained from  \eqref{airplane}. In Figure \ref{airplane_plot_fi} we can see that $\phi_1\leq\phi_2$ and that $\phi_2/\phi_1$ is increasing in $(0,\infty)$. Under these conditions, from Proposition \ref{Pmain}, we can assure that $T_1\leq_{ST} T_2$ and $T_1\leq_{HR}T_2$ hold for all $R$.
\end{example}

\begin{figure}
	\centering
	\includegraphics[width=\textwidth]{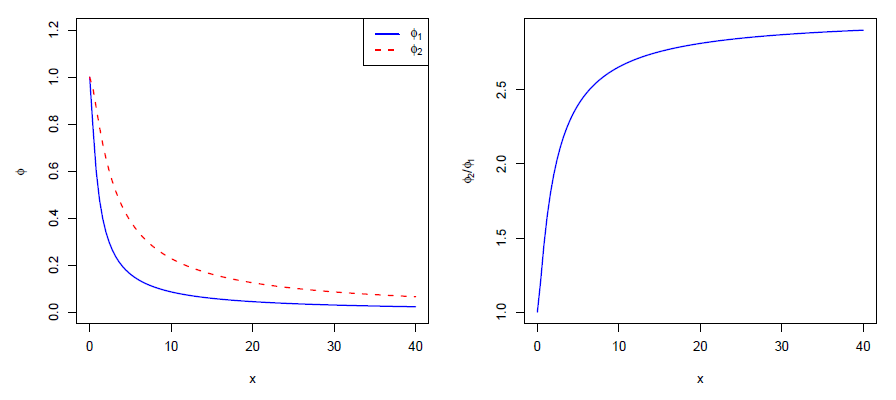}%{Example413_fis.pdf}
	\caption{Plots of $\phi_1$ and $\phi_2$ functions  (left) and ratio $\phi_2/\phi_1$ (right)  for the systems in Example \ref{Exairplane}.}\label{airplane_plot_fi}
\end{figure}

\section{Residual lifetimes}

In this section we study the system residual lifetime when the components satisfy the TTE model. 
It is well known that the reliability functions $\ov F_{i,t}$ of the component residual lifetimes $(X_i)_t = (X_i - t |X_i > t)$
can be written as $\ov F_{i,t}(x) =\ov  F_i(t +x)/\ov F_i(t)$, for $i = 1,\dots, n$, where we assume
$\ov F_i (t) > 0$. If the components are identically distributed (ID), then the common reliability function of $(X_i)_t$,
$i = 1,\dots,n$, will be represented simply as $\ov  F_t(x) =\ov  F(t + x)/\ov F(t)$, where we assume
$\ov  F(t) > 0$.

As in \cite{N18b}, if $T$ is the system lifetime, then we can consider two different residual lifetimes. The first one is the usual residual lifetime $T_t=(T-t|T>t)$ in which we just assume that, at time $t$, the system is working (i.e. $T>t$). The second one is called the residual lifetime at the system level and is defined by $T^*_t=(T-t|X_{1:n}>t)$. Here we assume that, at time $t$, all the components in the system are working (i.e. $X_{1:n}>t$). Note that both options coincide  when we consider the series system with $n$ components (i.e. when $T=X_{1:n}$).  However, for other system structures they are different and one may expect that $T^*_t$ is greater than $T_t$. This is true for the stochastic order when the components are independent but it is not always the case when we use other orders and/or the components are dependent \cite[see][]{N18b}.

The reliability function of $T_t$ can be obtained as
$$\overline{F}_{T_t}(x)=\Pr(T-t>x|T>t)=\frac{\overline{F}_T(x+t)}{\overline{F}_T(t)},$$
where $\overline{F}_T$ is given by \eqref{RT} under a TTE model. Analogously, the reliability function of  $T^*_t$  under a TTE model can be written as
$$\overline{F}_{T^*_t}(x)=\Pr(T-t>x|X_{1:n}>t)=\frac{\Pr(T>x+t, X_{1:n}>t)}{W(R_1(t)+\dots+R_n(t))},$$
where $\Pr(T>x+t, X_{1:n}>t)$ can be computed as $\overline{F}_T$ in  \eqref{RT}. Let us see some examples.

\begin{itemize}
\item[(i)] For a single component $X_i$, we have
$$\overline F_{i,t}(x) := \Pr(X_i-t>x|X_i>t)=\frac{\overline F_{i}(x+t)}{\overline F_{i}(t)}=\frac{W\left(R_i(x+t)\right)}{W\left(R_i(t)\right)}$$
and
\begin{align*}
\overline F^*_{i,t}(x) &:= \Pr(X_i-t>x|X_{1:n}>t)=
 \frac{W\left(R_i(x+t)+\sum_{j\neq i}R_j(t)\right)}{W(R_1(t)+\dots+R_n(t))}.
\end{align*}
\item[(ii)] 
 For the series system with components in $P$, that is, $X_P=\min_{j\in P}X_j$, we have
$$\overline F_{P,t}(x) := \Pr(X_P-t>x|X_P>t)=\frac{\overline F_{P}(x+t)}{\overline F_{P}(t)}=\frac{W\left(\sum_{j\in P}R_j(x+t)\right)}{W\left(\sum_{j\in P}R_j(t)\right)}$$
and
\begin{align*}
\overline F^*_{P,t}(x) &:= \Pr(X_P-t>x|X_{1:n}>t)=
\frac{W\left(\sum_{j\in P}R_j(x+t)+\sum_{j\notin P} R_j(t)\right)}{W(R_1(t)+\dots+R_n(t))}.
\end{align*}
In particular, for the series system with $n$ components, both options coincide and  we have 
$$			
\overline F_{1:n,t}(x):= \Pr(X_{1:n}-t>x|X_{1:n}>t)
=\frac{W(R_1(x+t)+\dots+ R_n(x+t))}{W(R_1(t)+\dots+R_n(t))}.
$$		
\item[(iii)] For a parallel system with two components we have 
\begin{align*}
\overline{F}_{2:2,t}(x):&=\Pr(X_{2:2}-t|X_{2:2}>t)\\
&=\frac{W(R_1(x+t))+W(R_2(x+t))-W(R_1(x+t)+R_2(x+t))}{W(R_1(t))+W(R_2(t))-W(R_1(t)+R_2(t))}
\end{align*}
and
\begin{align*}
\overline F^*_{2:2,t}(x) :&= \Pr(X_{2:2}-t>x|X_{1:2}>t)\\
&=\frac{W(R_1(x+t)+R_2(t))+W(R_1(t)+R_2(x+t))-W(R_1(x+t)+R_2(x+t))}{W(R_1(t)+R_2(t))}.
		\end{align*}
		%\item[(iv)] For a $2$-out-of-$3$ system we have
	\end{itemize}
%\end{example}

Let us give some comparison results. In the first one we compare the residual lifetimes of series systems.

\begin{proposition}
If $(X_1,\dots,X_n)\sim TTE(W,R_1,\dots,R_n)$, $T=X_P$ and $W$ satisfies 
\begin{equation}\label{TP2}
W(c)W(a+b)-W(a)W(c+b)\geq 0 \quad (\leq 0) \text{ for all }a\geq c\geq 0,b\geq 0,
\end{equation}
then $T_{t}\leq_{ST}T^*_{t}$ ($\geq_{ST}$) for all $P\subset \{1,\dots,n\}$. 
\end{proposition}
\begin{proof}
From the expressions given above this ordering holds iff 
$$\frac{W\left(\sum_{j\in P}R_j(x+t)\right)}{W\left(\sum_{j\in P}R_j(t)\right)}\leq 
\frac{W\left(\sum_{j\in P}R_j(x+t)+\sum_{j\notin P} R_j(t)\right)}{W(R_1(t)+\dots+R_n(t))}.$$
If we define
$$a:=\sum_{j\in P}R_j(x+t),$$
$$b:=\sum_{j\notin P} R_j(t)$$
and
$$c:=\sum_{j\in P}R_j(t),$$
this condition can be written as $W(c)W(a+b)-W(a)W(c+b)\geq 0$. As $R_j$ is nonnegative and  increasing, then we have $a\geq c\geq 0$ and $b\geq 0$. So this condition holds if $W$ satisfies \eqref{TP2}.
\end{proof}

\begin{remark}
In particular, the preceding proposition can be applied to $P=\{i\}$, that is, to a single component, when \eqref{TP2} holds. However, we do not need this condition when $P=\{1,\dots,n\}$ since $X_{1:n,t}=_{ST}X^*_{1:n,t}$.
Note that \eqref{TP2}   can be written as 
$$\frac{W(t+x)}{W(t)}\geq \frac{W(s+x)}{W(s)} \quad (\leq)  \text{ for all }0\leq s\leq t ,x\geq 0.$$
Thus, if we see $W$ as a reliability function, \eqref{TP2} is equivalent to the condition: $W$ is DFR (IFR) \cite[see, e.g.,][p. 15]{SS07}.
In particular it holds when the components are independent, i.e., when $W(x)=e^{-x}$ for $x\geq 0$. Actually, in this case we have $T_{t}=_{ST}T^*_{t}$  for all $P\subset \{1,\dots,n\}$ (as expected). If, for example, $W(x)=b^a(b+x)^{-a}$  for $x\geq 0$ and  $a,b>0$ (Clayton-Oakes copula), then $T_{t}\leq_{ST}T^*_{t}$ holds for all $t$ since $W$ is DFR.     Also note that we may obtain the reverse inequality $T_{t}\geq_{ST}T^*_{t}$ when $W$ is IFR. Let us see an example.
\end{remark}

\begin{example}\label{Ex5.3}
Let us consider $W(x)=\exp\left(\frac{1-e^x}{\theta}\right)$ for $x\geq 0$ with $\theta\in(0,1]$ which corresponds to the Gumbel-Barnett copula. It can be seen that $W$ is IFR (and convex).
If $(X_1,X_2,X_3)\sim\textup{TTE}(W,R_1,R_2,R_3)$ and we consider the series system $X_P$ with $P=\{1,2\}$, the respective reliability functions for the different system residual lifetimes are
$$\overline F_{P,t}(x) = \frac{W\left(R_1(x+t)+R_2(x+t)\right)}{W\left(R_1(t)+R_2(t)\right)}$$
and
$$\overline F^*_{P,t}(x) = \frac{W\left(R_1(x+t)+R_2(x+t)+R_3(t)\right)}{W\left(R_1(t)+R_2(t)+R_3(t)\right)}.$$
By plotting these reliability functions we can see if they are  ordered. For example, if $t=1$, $\theta=0.5$, $R_1(x)=(\exp(x)-1)/10$,   $R_2(x)=(\exp(x)-1)/5$, and  $R_3(x)=x$ for $x\geq 0$, 
we obtain the reliability functions  given in Figure \ref{figNew} (left). 	As expected we obtain that $T_{1}\geq_{ST}T^*_{1}$, that is, for this system the residual lifetime is greater if we just know that it is working at time $t=1$ than if we know that, at this time, all the components are working. Even more, as the ratio $\overline F^*_{P,1}/\overline F_{P,1}$ plotted in  Figure \ref{figNew} (left) is decreasing, we have $T_{1}\geq_{HR}T^*_{1}$.
\end{example}

\begin{figure}
	\centering
	\includegraphics[width=\textwidth]{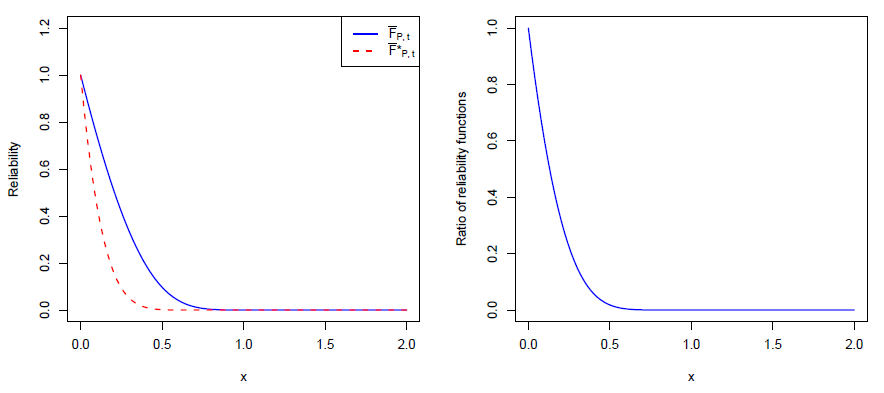}%{remark52.pdf}
	\caption{Reliability functions (left) and ratio $\overline F^*_{P,1}/\overline F_{P,1}$ (right) for the residual lifetimes of the system in Example \ref{Ex5.3}}\label{figNew}
\end{figure}

\section{Conclusions}

The present paper provide the basic results for the theoretical study of coherent systems with dependent components under a TTE model. This includes explicit expressions for the system reliability functions and tools to compare different systems and/or a system whose components have different dependence TTE models. The system residual lifetimes under this model are studied as well. 

This paper is just a first (basic) step. Clearly, the second one is to study how to apply these results in practice. Here, maybe, the main task will be how to estimate the distribution of the risk parameter $\theta$ which determines the frailty dependence model from  system lifetimes  and/or component lifetimes data. Moreover, we could study if the positive (or negative) dependency is good or bad for a specific system structure (as in Example \ref{Exairplane}).
The results included in the present paper are basic tools for those purposes.
 
\section*{Acknowledgements}
%\acks

%We would like to thank an anonymous reviewer for several helpful
%suggestions that have served to add clarity and breadth to the
%earlier version of this paper.

JN and JM  acknowledge the support received from Ministerio de Econom\'ia, Industria y
Competitividad of Spain under grant {\it MTM2016-79943-P} (AEI/ FEDER, UE).
JM also acknowledges the support received from the Conselleria d'Educaci\'o, Investigaci\'o, Cultura i Esport (Generalitat de la Comunitat Valenciana).

\end{document}